\documentclass{article}
\usepackage{amsmath}

\newtheorem{theorem}{Theorem}

\newenvironment{proof}[1][Proof]{\noindent\textbf{#1.} }{\ \rule{0.5em}{0.5em}}
\input{tcilatex}

\begin{document}

\title{On the Fibonacci and Lucas numbers, their sums and permanents of one
type of Hessenberg matrices}
\author{{\small \ Fatih YILMAZ\thanks{%
e-mail adresses: fyilmaz@selcuk.edu.tr, dbozkurt@selcuk.edu.tr} and Durmus
BOZKURT} \\
{\small Selcuk University, Science Faculty Department of Mathematics, 42250 }%
\linebreak \\
{\small Campus} {\small Konya, Turkey}}
\maketitle

\begin{abstract}
At this paper,$\ $we derive some relationships between permanents of one
type of lower-Hessenberg matrix and the Fibonacci and Lucas\ numbers and
their sums.
\end{abstract}

\section{Introduction}

The well-known Fibonacci and Lucas sequences are recursively defined by 
\begin{eqnarray*}
F_{n+1} &=&F_{n}+F_{n-1},\text{ }n\geq 1 \\
L_{n+1} &=&L_{n}+L_{n-1},\text{ }n\geq 1
\end{eqnarray*}%
with initial conditions $F_{0}=0,$\ $F_{1}=1$ and $L_{0}=2,$ $L_{1}=1.$ The
first few values of the sequences are given below:%
\begin{equation*}
\begin{tabular}{c|cccccccccc}
$n$ & $0$ & $1$ & $2$ & $3$ & $4$ & $5$ & $6$ & $7$ & $8$ & $9$ \\ \hline
$F_{n}$ & $0$ & $1$ & $1$ & $2$ & $3$ & $5$ & $8$ & $13$ & $21$ & $34$ \\ 
$L_{n}$ & $2$ & $1$ & $3$ & $4$ & $7$ & $11$ & $18$ & $29$ & $47$ & $76$%
\end{tabular}%
\end{equation*}

The permanent of a matrix is similar to the determinant but all of the signs
used in the Laplace expansion of minors are positive. The permanent of an $n$%
-square matrix is defined by%
\begin{equation*}
perA=\underset{\sigma \in S_{n}}{\dsum }\underset{i=1}{\overset{n}{\dprod }}%
a_{i\sigma (i)}
\end{equation*}%
where the summation extends over all permutations $\sigma $ of the symmetric
group $S_{n}$ $\cite{5}.$

Let $A=[a_{ij}]$ be an $m\times n$ matrix with row vectors $%
r_{1},r_{2},\ldots ,r_{m}.$ We call $A$ is \textit{contractible} on column $%
k $, if column $k$ contains exactly two non zero elements. Suppose that $A$
is contractible on column $k$ with $a_{ik}\neq 0,a_{jk}\neq 0$ and $i\neq j.$
Then the $(m-1)\times (n-1)$ matrix $A_{ij:k}$ obtained from $A$ replacing
row $i$ with $a_{jk}r_{i}+a_{ik}r_{j}$ and deleting row $j$ and column $k$
is called the \textit{contraction} of $A$ on column $k$ relative to rows $i$
and $j$. If $A$ is contractible on row $k$ with $a_{ki}\neq 0,a_{kj}\neq 0$
and $i\neq j,$ then the matrix $A_{k:ij}=[A_{ij:k}^{T}]^{T}$ is called the
contraction of $A$ on row $k$ relative to columns $i$ and $j$. We know that
if $A$ is a nonnegative matrix and $B$ is a contraction of $A~\cite{2}$,
then 
\begin{equation}
perA=perB.  \label{10}
\end{equation}

It is known that there are a lot of relations between determinants or
\linebreak permanents of matrices and well-known number sequences. For
example, the authors $\cite{2}$ investigate relationships between permanents
of one type of \linebreak Hessenberg matrix the Pell and Perrin numbers.

In $\cite{3},$ Lee defined the matrix 
\begin{equation*}
\pounds _{n}=\left[ 
\begin{array}{cccccc}
1 & 0 & 1 & 0 & \cdots & 0 \\ 
1 & 1 & 1 & 0 & \cdots & 0 \\ 
0 & 1 & 1 & 1 &  & \vdots \\ 
0 & 0 & 1 & 1 & \ddots & 0 \\ 
\vdots & \vdots &  & \ddots & \ddots & 1 \\ 
0 & 0 & \cdots & 0 & 1 & 1%
\end{array}%
\right]
\end{equation*}%
and showed that 
\begin{equation*}
per(\pounds _{n})=L_{n-1}
\end{equation*}%
where $L_{n}$ is the $n$th Lucas number.

In $\cite{4},$ the author investigate general tridiagonal matrix
determinants and permanents. Also he showed that the permanent of the
tridiagonal matrix based on $\{a_{i}\}$, $\{b_{i}\},$ $\{c_{i}\}$ is equal
to the determinant of the matrix based on \linebreak $\{-a_{i}\}$, $%
\{b_{i}\},$ $\{c_{i}\}.$

In $\cite{7},~$the authors give $(0,1,-1)$ tridiagonal matrices whose
determinants and permanents are negatively subscripted Fibonacci and Lucas
numbers. Also, they give an $n\times n$ $(-1,1)$ matrix $S$,%
\begin{equation*}
S=\left[ 
\begin{array}{rrrrc}
1 & 1 & \cdots & 1 & 1 \\ 
-1 & 1 & \cdots & 1 & 1 \\ 
1 & -1 & \cdots & 1 & 1 \\ 
\vdots & \vdots & \ddots & \vdots & \vdots \\ 
1 & 1 & \cdots & -1 & 1%
\end{array}%
\right] .
\end{equation*}%
such that per$A$=det$(A\circ S$), where $A\circ S$ denotes Hadamard product
of $A$ and $S$.

In the present paper, we consider a particular case of lower Hessenberg
\linebreak matrices. Then, we show that the permanents of these type of
matrices are related with Fibonacci and Lucas numbers and their sums.

\section{Determinantal representation of Fibonacci and \protect\linebreak %
Lucas numbers and their sums}

In this section, we define one type of lower Hessenberg matrix and show that
the permanents of these type of matrices are Fibonacci, Lucas numbers and
their sums.

Let $H_{n}=[h_{ij}]_{n\times n}$ be an $n$-square lower Hessenberg matrix in
which the superdiagonal entries are alternating $-1$s and $1$s, starting
with $-1$, the main diagonal entries are $2$s, except the last one which is $%
1$, the subdiagonal entries are $0$s, the lower-subdiagonal entries are $1$s
and otherwise $0$. That is:

\begin{equation}
H_{n}=\left[ 
\begin{array}{ccccccc}
2 & -1 &  &  &  &  &  \\ 
0 & 2 & 1 &  &  &  &  \\ 
1 & 0 & 2 & -1 &  &  &  \\ 
& 1 & 0 & 2 & 1 &  &  \\ 
&  & \ddots & \ddots & \ddots & \ddots &  \\ 
&  &  & 1 & 0 & 2 & (-1)^{n-1} \\ 
&  &  &  & 1 & 0 & 1%
\end{array}%
\right]  \label{30}
\end{equation}

\begin{theorem}
Let $H_{n}$ be as in (\ref{30}), then 
\begin{equation*}
perH_{n}=perH_{n}^{(n-2)}=F_{n+1}
\end{equation*}%
where $F_{n}$ is the $n$th Fibonacci number.
\end{theorem}

\begin{proof}
By definition of the matrix $H_{n},$ it can be contracted on column $n$. Let 
$H_{n}^{(r)}$ be the $r$th contraction of $H_{n}$. If $r=1$, then 
\begin{equation*}
H_{n}^{(1)}=\left[ 
\begin{array}{ccccccc}
2 & -1 &  &  &  &  &  \\ 
0 & 2 & 1 &  &  &  &  \\ 
1 & 0 & 2 & -1 &  &  &  \\ 
& 1 & 0 & 2 & 1 &  &  \\ 
&  & \ddots & \ddots & \ddots & \ddots &  \\ 
&  &  & 1 & 0 & 2 & (-1)^{n-2} \\ 
&  &  &  & 1 & (-1)^{n-1} & 2%
\end{array}%
\right] .
\end{equation*}%
Since $H_{n}^{(1)}$ also can be contracted according to the last column,%
\begin{equation*}
H_{n}^{(2)}=\left[ 
\begin{array}{ccccccc}
2 & -1 &  &  &  &  &  \\ 
0 & 2 & 1 &  &  &  &  \\ 
1 & 0 & 2 & -1 &  &  &  \\ 
& 1 & 0 & 2 & 1 &  &  \\ 
&  & \ddots & \ddots & \ddots & \ddots &  \\ 
&  &  & 1 & 0 & 2 & (-1)^{n-3} \\ 
&  &  &  & 2 & (-1)^{n-2} & 3%
\end{array}%
\right] .
\end{equation*}%
Furthermore, the matrix $H_{n}^{(2)}$\ can be contracted on the last column,
that is%
\begin{equation*}
H_{n}^{(3)}=\left[ 
\begin{array}{ccccccc}
2 & -1 &  &  &  &  &  \\ 
0 & 2 & 1 &  &  &  &  \\ 
1 & 0 & 2 & -1 &  &  &  \\ 
& 1 & 0 & 2 & 1 &  &  \\ 
&  & \ddots & \ddots & \ddots & \ddots &  \\ 
&  &  & 1 & 0 & 2 & (-1)^{n-4} \\ 
&  &  &  & 3 & (-1)^{n-3}2 & 5%
\end{array}%
\right] .
\end{equation*}%
Continuing this method, we obtain the $r$th contraction%
\begin{eqnarray*}
H_{n}^{(r)} &=&\left[ 
\begin{array}{ccccccc}
2 & -1 &  &  &  &  &  \\ 
0 & 2 & 1 &  &  &  &  \\ 
1 & 0 & 2 & -1 &  &  &  \\ 
& 1 & 0 & 2 & 1 &  &  \\ 
&  & \ddots & \ddots & \ddots & \ddots &  \\ 
&  &  & 1 & 0 & 2 & (-1)^{r-1} \\ 
&  &  &  & F_{r+1} & (-1)^{r}(F_{r+2}-F_{r+1}) & F_{r+2}%
\end{array}%
\right] ,\text{ }n\text{ is even} \\
H_{n}^{(r)} &=&\left[ 
\begin{array}{ccccccc}
2 & -1 &  &  &  &  &  \\ 
0 & 2 & 1 &  &  &  &  \\ 
1 & 0 & 2 & -1 &  &  &  \\ 
& 1 & 0 & 2 & 1 &  &  \\ 
&  & \ddots & \ddots & \ddots & \ddots &  \\ 
&  &  & 1 & 0 & 2 & (-1)^{r} \\ 
&  &  &  & F_{r+1} & (-1)^{r-1}(F_{r+2}-F_{r+1}) & F_{r+2}%
\end{array}%
\right] ,\text{ }n\text{ is odd}
\end{eqnarray*}%
where $2\leq r\leq n-4.$ Hence 
\begin{equation*}
H_{n}^{(n-3)}=\left[ 
\begin{array}{ccc}
2 & -1 & 0 \\ 
0 & 2 & 1 \\ 
F_{n-2} & (F_{n-2}-F_{n-1}) & F_{n-1}%
\end{array}%
\right] \text{ }
\end{equation*}%
which by contraction of $H_{n}^{(n-3)}$ on column $3$, 
\begin{equation*}
H_{n}^{(n-2)}=\left[ 
\begin{array}{cc}
2 & -1 \\ 
F_{n-2} & F_{n}%
\end{array}%
\right] .
\end{equation*}%
By (\ref{10}), we have $perH_{n}=perH_{n}^{(n-2)}=F_{n+1}$.
\end{proof}

Let $K_{n}=[k_{ij}]_{n\times n}$ be an $n$-square lower Hessenberg matrix in
which the superdiagonal entries are alternating $-1$s and $1$s starting with 
$1$, except the first one which is $-3$, the main diagonal entries are $2$s,
except the last one which is $1$, the subdiagonal entries are $0$s, the
lower-subdiagonal entries are $1$s and otherwise $0$. Clearly:

\begin{equation}
K_{n}=\left[ 
\begin{array}{ccccccc}
2 & -3 &  &  &  &  &  \\ 
0 & 2 & 1 &  &  &  &  \\ 
1 & 0 & 2 & -1 &  &  &  \\ 
& 1 & 0 & 2 & 1 &  &  \\ 
&  & \ddots & \ddots & \ddots & \ddots &  \\ 
&  &  & 1 & 0 & 2 & (-1)^{n-1} \\ 
&  &  &  & 1 & 0 & 1%
\end{array}%
\right]  \label{40}
\end{equation}

\begin{theorem}
Let $K_{n}$ be as in (\ref{40}), then 
\begin{equation*}
perK_{n}=perK_{n}^{(n-2)}=L_{n-2}
\end{equation*}%
where $L_{n}$ is the $n$th Lucas number.
\end{theorem}

\begin{proof}
By definition of the matrix $K_{n}$, it can be contracted on column $n$.
That is, 
\begin{equation*}
K_{n}^{(1)}=\left[ 
\begin{array}{ccccccc}
2 & -3 &  &  &  &  &  \\ 
0 & 2 & 1 &  &  &  &  \\ 
1 & 0 & 2 & -1 &  &  &  \\ 
& 1 & 0 & 2 & 1 &  &  \\ 
&  & \ddots & \ddots & \ddots & \ddots &  \\ 
&  &  & 1 & 0 & 2 & (-1)^{n-2} \\ 
&  &  &  & 1 & (-1)^{n-1} & 2%
\end{array}%
\right] .
\end{equation*}%
$K_{n}^{(1)}$ also can be contracted on the last column,%
\begin{equation*}
K_{n}^{(2)}=\left[ 
\begin{array}{ccccccc}
2 & -3 &  &  &  &  &  \\ 
0 & 2 & 1 &  &  &  &  \\ 
1 & 0 & 2 & -1 &  &  &  \\ 
& 1 & 0 & 2 & 1 &  &  \\ 
&  & \ddots & \ddots & \ddots & \ddots &  \\ 
&  &  & 1 & 0 & 2 & (-1)^{n-3} \\ 
&  &  &  & 2 & (-1)^{n-2} & 3%
\end{array}%
\right] .
\end{equation*}%
$K_{n}^{(2)}$ also can be contracted on the last column,%
\begin{equation*}
K_{n}^{(3)}=\left[ 
\begin{array}{ccccccc}
2 & -3 &  &  &  &  &  \\ 
0 & 2 & 1 &  &  &  &  \\ 
1 & 0 & 2 & -1 &  &  &  \\ 
& 1 & 0 & 2 & 1 &  &  \\ 
&  & \ddots & \ddots & \ddots & \ddots &  \\ 
&  &  & 1 & 0 & 2 & (-1)^{n-4} \\ 
&  &  &  & 3 & 2(-1)^{n-3} & 5%
\end{array}%
\right] .
\end{equation*}%
Going with this process, we have%
\begin{eqnarray*}
K_{n}^{(r)} &=&\left[ 
\begin{array}{ccccccc}
2 & -3 &  &  &  &  &  \\ 
0 & 2 & 1 &  &  &  &  \\ 
1 & 0 & 2 & -1 &  &  &  \\ 
& 1 & 0 & 2 & 1 &  &  \\ 
&  & \ddots & \ddots & \ddots & \ddots &  \\ 
&  &  & 1 & 0 & 2 & (-1)^{r-1} \\ 
&  &  &  & F_{r+1} & (-1)^{r-2}(F_{r+2}-F_{r+1}) & F_{r+2}%
\end{array}%
\right] ,\text{ }n\text{ is even} \\
K_{n}^{(r)} &=&\left[ 
\begin{array}{ccccccc}
2 & -3 &  &  &  &  &  \\ 
0 & 2 & 1 &  &  &  &  \\ 
1 & 0 & 2 & -1 &  &  &  \\ 
& 1 & 0 & 2 & 1 &  &  \\ 
&  & \ddots & \ddots & \ddots & \ddots &  \\ 
&  &  & 1 & 0 & 2 & (-1)^{r} \\ 
&  &  &  & F_{r+1} & (-1)^{r-1}(F_{r+2}-F_{r+1}) & F_{r+2}%
\end{array}%
\right] ,\text{ }n\text{ is odd}
\end{eqnarray*}%
for $2\leq r\leq n-4.$ Hence 
\begin{equation*}
K_{n}^{(n-3)}=\left[ 
\begin{array}{ccc}
2 & -3 & 0 \\ 
0 & 2 & 1 \\ 
F_{n-3} & F_{n-3}-F_{n-1} & F_{n-1}%
\end{array}%
\right]
\end{equation*}%
which by contraction of $K_{n}^{(n-3)}$ on column $3$, gives%
\begin{equation*}
K_{n}^{(n-2)}=\left[ 
\begin{array}{cc}
2 & -3 \\ 
F_{n-2} & F_{n}%
\end{array}%
\right] .
\end{equation*}%
By applying (\ref{10}), we have $%
perK_{n}=perK_{n}^{(n-2)}=2F_{n}-3F_{n-2}=L_{n-2},~$which is desired.
\end{proof}

Let $M_{n}=[m_{ij}]_{n\times n}$ be an $n$-square lower Hessenberg matrix in
which the superdiagonal entries are alternating $-1$s and $1$s, starting
with $-1$, the main diagonal entries are $2$s, the subdiagonal entries are $%
0 $s, the lower-subdiagonal entries are $1$s and otherwise $0$. In other
words:

\begin{equation}
M_{n}=\left[ 
\begin{array}{ccccccc}
2 & -1 &  &  &  &  &  \\ 
0 & 2 & 1 &  &  &  &  \\ 
1 & 0 & 2 & -1 &  &  &  \\ 
& 1 & 0 & 2 & 1 &  &  \\ 
&  & \ddots & \ddots & \ddots & \ddots &  \\ 
&  &  & 1 & 0 & 2 & (-1)^{n-1} \\ 
&  &  &  & 1 & 0 & 2%
\end{array}%
\right]  \label{50}
\end{equation}

\begin{theorem}
Let $M_{n}$ be as in (\ref{50}), then 
\begin{equation*}
perM_{n}=perM_{n}^{(n-2)}=\underset{i=0}{\overset{n-1}{\dsum }}F_{i}
\end{equation*}%
where $F_{n}$ is the $n$th Fibonacci number.
\end{theorem}

\begin{proof}
By definition of the matrix $M_{n}$, it can be contracted on column $n$.
That is, 
\begin{equation*}
M_{n}^{(1)}=\left[ 
\begin{array}{ccccccc}
2 & -1 &  &  &  &  &  \\ 
0 & 2 & 1 &  &  &  &  \\ 
1 & 0 & 2 & -1 &  &  &  \\ 
& 1 & 0 & 2 & 1 &  &  \\ 
&  & \ddots & \ddots & \ddots & \ddots &  \\ 
&  &  & 1 & 0 & 2 & (-1)^{n-2} \\ 
&  &  &  & 2 & (-1)^{n-1} & 4%
\end{array}%
\right] .
\end{equation*}%
$M_{n}^{(1)}$ also can be contracted on the last column,%
\begin{equation*}
M_{n}^{(2)}=\left[ 
\begin{array}{ccccccc}
2 & -1 &  &  &  &  &  \\ 
0 & 2 & 1 &  &  &  &  \\ 
1 & 0 & 2 & -1 &  &  &  \\ 
& 1 & 0 & 2 & 1 &  &  \\ 
&  & \ddots & \ddots & \ddots & \ddots &  \\ 
&  &  & 1 & 0 & 2 & (-1)^{n-3} \\ 
&  &  &  & 4 & 2(-1)^{n-2} & 7%
\end{array}%
\right] .
\end{equation*}%
$M_{n}^{(2)}$ also can be contracted on the last column,%
\begin{equation*}
M_{n}^{(3)}=\left[ 
\begin{array}{ccccccc}
2 & -1 &  &  &  &  &  \\ 
0 & 2 & 1 &  &  &  &  \\ 
1 & 0 & 2 & -1 &  &  &  \\ 
& 1 & 0 & 2 & 1 &  &  \\ 
&  & \ddots & \ddots & \ddots & \ddots &  \\ 
&  &  & 1 & 0 & 2 & (-1)^{n-4} \\ 
&  &  &  & 7 & 4(-1)^{n-3} & 12%
\end{array}%
\right] .
\end{equation*}%
Going with this process, we have%
\begin{eqnarray*}
M_{n}^{(r)} &=&\left[ 
\begin{array}{ccccccc}
2 & -1 &  &  &  &  &  \\ 
0 & 2 & 1 &  &  &  &  \\ 
1 & 0 & 2 & -1 &  &  &  \\ 
& 1 & 0 & 2 & 1 &  &  \\ 
&  & \ddots & \ddots & \ddots & \ddots &  \\ 
&  &  & 1 & 0 & 2 & (-1)^{r} \\ 
&  &  &  & \underset{i=0}{\overset{r+1}{\dsum }}F_{i} & (-1)^{r-1}\underset{%
i=0}{\overset{r}{\dsum }}F_{i} & \underset{i=0}{\overset{r+2}{\dsum }}F_{i}%
\end{array}%
\right] ,~n\text{ is odd} \\
M_{n}^{(r)} &=&\left[ 
\begin{array}{ccccccc}
2 & -1 &  &  &  &  &  \\ 
0 & 2 & 1 &  &  &  &  \\ 
1 & 0 & 2 & -1 &  &  &  \\ 
& 1 & 0 & 2 & 1 &  &  \\ 
&  & \ddots & \ddots & \ddots & \ddots &  \\ 
&  &  & 1 & 0 & 2 & (-1)^{r-1} \\ 
&  &  &  & \underset{i=0}{\overset{r+1}{\dsum }}F_{i} & (-1)^{r-2}\underset{%
i=0}{\overset{r}{\dsum }}F_{i} & \underset{i=0}{\overset{r+2}{\dsum }}F_{i}%
\end{array}%
\right] ,~n\text{ is even}
\end{eqnarray*}%
for $2\leq r\leq n-4.$ Hence 
\begin{equation*}
M_{n}^{(n-3)}=\left[ 
\begin{array}{ccc}
2 & -1 & 0 \\ 
0 & 2 & 1 \\ 
\underset{i=0}{\overset{n-2}{\dsum }}F_{i} & -\underset{i=0}{\overset{n-3}{%
\dsum }}F_{i} & \underset{i=0}{\overset{n-1}{\dsum }}F_{i}%
\end{array}%
\right]
\end{equation*}%
which by contraction of $M_{n}^{(n-3)}$ on column $3$, gives%
\begin{equation*}
M_{n}^{(n-2)}=\left[ 
\begin{array}{cc}
2 & -1 \\ 
\underset{i=0}{\overset{n-4}{\dsum }}F_{i} & \underset{i=0}{\overset{n-2}{%
\dsum }}F_{i}%
\end{array}%
\right] .
\end{equation*}%
By applying (\ref{10}), we have%
\begin{equation*}
perM_{n}=perM_{n}^{(n-2)}=\underset{i=0}{\overset{n-1}{\dsum }}F_{i}
\end{equation*}%
which is desired.
\end{proof}

Let $N_{n}=[n_{ij}]_{n\times n}$ be an $n$-square lower Hessenberg matrix in
which the superdiagonal entries are alternating $-1$s and $1$s starting with 
$1$, except the first one which is $-2$, the main diagonal entries are $2$s,
except the first one is $3$, the subdiagonal entries are $0$s, the
lower-subdiagonal entries are $1$s and otherwise $0$. In this content:

\begin{equation}
N_{n}=\left[ 
\begin{array}{ccccccc}
3 & -2 &  &  &  &  &  \\ 
0 & 2 & 1 &  &  &  &  \\ 
1 & 0 & 2 & -1 &  &  &  \\ 
& 1 & 0 & 2 & 1 &  &  \\ 
&  & \ddots & \ddots & \ddots & \ddots &  \\ 
&  &  & 1 & 0 & 2 & (-1)^{n-1} \\ 
&  &  &  & 1 & 0 & 2%
\end{array}%
\right]  \label{60}
\end{equation}

\begin{theorem}
Let $N_{n}$ be an $n$-square matrix $(n\geq 2)~$as in (\ref{60}), then 
\begin{equation*}
perN_{n}=perN_{n}^{(n-2)}=\underset{i=0}{\overset{n}{\dsum }}L_{i}
\end{equation*}%
where $L_{n}$ is the $n$th Lucas number.
\end{theorem}

\begin{proof}
By definition of the matrix $N_{n}$, it can be contracted on column $n$.
That is, 
\begin{equation*}
N_{n}^{(1)}=\left[ 
\begin{array}{ccccccc}
3 & -2 &  &  &  &  &  \\ 
0 & 2 & 1 &  &  &  &  \\ 
1 & 0 & 2 & -1 &  &  &  \\ 
& 1 & 0 & 2 & 1 &  &  \\ 
&  & \ddots & \ddots & \ddots & \ddots &  \\ 
&  &  & 1 & 0 & 2 & (-1)^{n} \\ 
&  &  &  & 2 & (-1)^{n-1} & 4%
\end{array}%
\right] .
\end{equation*}%
$N_{n}^{(1)}$ also can be contracted on the last column,%
\begin{equation*}
N_{n}^{(2)}=\left[ 
\begin{array}{ccccccc}
3 & -2 &  &  &  &  &  \\ 
0 & 2 & 1 &  &  &  &  \\ 
1 & 0 & 2 & -1 &  &  &  \\ 
& 1 & 0 & 2 & 1 &  &  \\ 
&  & \ddots & \ddots & \ddots & \ddots &  \\ 
&  &  & 1 & 0 & 2 & (-1)^{n-1} \\ 
&  &  &  & 4 & 2(-1)^{n-2} & 7%
\end{array}%
\right] .
\end{equation*}%
$N_{n}^{(2)}$ also can be contracted on the last column,%
\begin{equation*}
N_{n}^{(3)}=\left[ 
\begin{array}{ccccccc}
3 & -2 &  &  &  &  &  \\ 
0 & 2 & 1 &  &  &  &  \\ 
1 & 0 & 2 & -1 &  &  &  \\ 
& 1 & 0 & 2 & 1 &  &  \\ 
&  & \ddots & \ddots & \ddots & \ddots &  \\ 
&  &  & 1 & 0 & 2 & (-1)^{n-2} \\ 
&  &  &  & 7 & 4(-1)^{n-3} & 12%
\end{array}%
\right] .
\end{equation*}%
Going with this process, we have%
\begin{eqnarray*}
N_{n}^{(r)} &=&\left[ 
\begin{array}{ccccccc}
3 & -2 &  &  &  &  &  \\ 
0 & 2 & 1 &  &  &  &  \\ 
1 & 0 & 2 & -1 &  &  &  \\ 
& 1 & 0 & 2 & 1 &  &  \\ 
&  & \ddots & \ddots & \ddots & \ddots &  \\ 
&  &  & 1 & 0 & 2 & (-1)^{r} \\ 
&  &  &  & \underset{i=0}{\overset{r+1}{\dsum }}F_{i} & (-1)^{r-1}\underset{%
i=0}{\overset{r}{\dsum }}F_{i} & \underset{i=0}{\overset{r+2}{\dsum }}F_{i}%
\end{array}%
\right] ,~n\text{ is odd} \\
N_{n}^{(r)} &=&\left[ 
\begin{array}{ccccccc}
3 & -2 &  &  &  &  &  \\ 
0 & 2 & 1 &  &  &  &  \\ 
1 & 0 & 2 & -1 &  &  &  \\ 
& 1 & 0 & 2 & 1 &  &  \\ 
&  & \ddots & \ddots & \ddots & \ddots &  \\ 
&  &  & 1 & 0 & 2 & (-1)^{r-1} \\ 
&  &  &  & \underset{i=0}{\overset{r+1}{\dsum }}F_{i} & (-1)^{r}\underset{i=0%
}{\overset{r}{\dsum }}F_{i} & \underset{i=0}{\overset{r+2}{\dsum }}F_{i}%
\end{array}%
\right] ,~n\text{ is even}
\end{eqnarray*}%
for $2\leq r\leq n-4.$ Hence 
\begin{equation*}
N_{n}^{(n-3)}=\left[ 
\begin{array}{ccc}
3 & -2 & 0 \\ 
0 & 2 & 1 \\ 
\underset{i=0}{\overset{n-2}{\dsum }}F_{i} & -\underset{i=0}{\overset{n-3}{%
\dsum }}F_{i} & \underset{i=0}{\overset{n-1}{\dsum }}F_{i}%
\end{array}%
\right]
\end{equation*}%
which by contraction of $N_{n}^{(n-3)}$ on column $3$, gives%
\begin{equation*}
N_{n}^{(n-2)}=\left[ 
\begin{array}{cc}
3 & -2 \\ 
\underset{i=0}{\overset{n-2}{\dsum }}F_{i} & \underset{i=0}{\overset{n}{%
\dsum }}F_{i}%
\end{array}%
\right] .
\end{equation*}%
By applying (\ref{10}), we have 
\begin{eqnarray*}
perN_{n} &=&perN_{n}^{(n-2)} \\
&=&F_{n+1}+\underset{i=0}{\overset{n+1}{\dsum }}F_{i} \\
&=&\underset{i=0}{\overset{n}{\dsum }}L_{i}
\end{eqnarray*}%
which is desired.
\end{proof}

\end{document}